\documentclass[11pt,a4paper,reqno]{amsart}
\usepackage[utf8]{inputenc}
\usepackage[main=english,french]{babel}
\usepackage{geometry}       
\usepackage[all]{xy}
\usepackage{amsmath,amsfonts,amssymb,amsthm}
\usepackage{mathrsfs}          

\usepackage{color}
\definecolor{bluegreen2}{RGB}{0, 85, 127}

\usepackage[linktocpage]{hyperref} 
\hypersetup{colorlinks=true,linkcolor=blue,citecolor=bluegreen2,filecolor=magenta,urlcolor=blue}

\usepackage{mysymbolsthms}
\usepackage{multiabstract}

\AtBeginDocument{}

\usepackage{color}
\usepackage[pdftex]{graphicx}
\usepackage{tikz, tikz-cd, pgfplots}

\title{On the equality of periods of Kontsevich-Zagier}

\author{Jacky Cresson}
\address{
Laboratoire de Math\'ematiques et de leurs Applications\newline\indent
UMR CNRS 5142\newline\indent
B\^atiment IPRA - Universit\'e de Pau et des Pays de l'Adour\newline\indent
Avenue de l'Universit\'e - BP 1155\newline\indent
64013 Pau, FRANCE
}
\email{jacky.cresson@univ-pau.fr}

\author{Juan Viu-Sos}
\address{
	Dpto. de Matemáticas e Informática,\newline\indent
    ETSI Caminos Canales y Puertos,\newline\indent
    Universidad Politécnica de Madrid,\newline\indent
    C\textbackslash Prof. Aranguren 3, 28040 Madrid, Spain
}
\email{jviusos@math.cnrs.fr}
\urladdr{https://jviusos.github.io/}

\thanks{The second named author is partially supported by the Spanish Government PID2020-114750GB-C32.}

\subjclass[2010]{Primary 11J81, 51M25; Secondary 52B45, 14P10}		

\date{}

\begin{document}

\begin{abstracts}
	\abstractin{english}
	Effective periods were defined by Kontsevich and Zagier as complex numbers whose real and imaginary parts are values of absolutely convergent integrals of $\QQ$-rational functions over $\QQ$-semi-algebraic domains in $\RR^d$. The Kontsevich-Zagier period conjecture states that any two different integral expressions of a period are related by a finite sequence of transformations only using three rules respecting the rationality of functions and domains: integral addition by integrands or domains, change of variables and Stokes' formula.

	In this paper, we introduce two geometric interpretations of this conjecture, seen as a generalization of Hilbert's third problem involving either compact semi-algebraic sets or rational polyhedra equipped with piece-wise algebraic forms. Based on known partial results for analogous Hilbert's third problems, we study possible geometric schemes to prove this conjecture and their potential obstructions.
	\bigskip
	
	\abstractin{french}
	Les périodes effectives furent définies par Kontsevich et Zagier comme étant les nombres complexes dont les parties réelle et imaginaire sont valeurs d'intégrales absolument convergentes de fonctions $\QQ$-rationnelles sur des domaines $\QQ$-semi-algébriques dans $\RR^d$. La conjecture des périodes de Kontsevich-Zagier affirme que si une période admet deux représentations intégrales, alors elles sont reliées par une suite finie d'opérations en utilisant uniquement trois règles respectant la rationalité des fonctions et domaines: sommes d'intégrales par intégrandes ou domaines, changement de variables et formule de Stokes.
	
	Dans cet article, nous introduissons deux interprétations géométriques de cette conjecture, vue comme une généralisation du 3\'eme problème de Hilbert soit pour des ensembles semi-algébriques compacts soit pour des polyèdres rationnels munis d'une forme volume algébrique par morceaux. Basés sur des résultats partiels connus pour des problèmes de Hilbert analogues, nous étudions
	des possibles schémas géométriques pour obtenir une preuve de la conjecture et ses obstructions potentielles.
\end{abstracts}

\selectlanguage{english}
\maketitle

\vspace*{-2em}
\tableofcontents

\pagebreak

\section{Introduction}\label{sec:intro}

Introduced by M.~Kontsevich and D.~Zagier in their paper \cite{KonZag01} in 2001, \emph{periods} are a class of numbers which contains most of the important constants in mathematics. They are strongly related to the study of transcendence in number theory \cite{Wald06}, Galois theory and motives (\cite{Andre04}, \cite{Andre12}, \cite{Ayoub15}) and differential equations~\cite{FisherRivoal14}. We refer to~\cite{Wald15} and~\cite{Muller14} for an overview of the subject.\\

Let $\QQbar$ (resp. $\RRalg$) be the \emph{field of complex} (resp. \emph{real}) \emph{algebraic numbers}. Following the \emph{affine} definition given in~\cite{KonZag01},
\emph{a period of Kontsevich-Zagier} (also called \emph{effective period}) is a complex number whose real and imaginary parts are values of absolutely convergent integrals of rational functions over domains in a real affine space given by polynomial inequalities both with coefficients in $\RRalg$ , i.e. absolutely convergent integrals of the form
\begin{equation}\label{eqn:int_form}
    \Ical(S,P/Q)=\int_S \frac{P(x_1,\ldots,x_d)}{Q(x_1,\ldots,x_d)}\cdot \dint x_1\wedge\cdots\wedge \dint x_d
\end{equation}
where $S\subset\RR^d$ is a $d$-dimensional $\RRalg$-semi-algebraic set and $P,Q\in\RRalg[x_1,\ldots,x_d]$ are coprime. We denote by $\Pkz$ the set of periods of Kontsevich-Zagier and by $\Pkz^\RR=\Pkz\cap\RR$ the set of {\em real periods}. These numbers are \emph{constructible}, in the sense that a period is directly associated with a set of integrands and domains of integrations given by polynomials of rational coefficients.
The set $\Pkz$ forms a constructible countable $\QQbar$-algebra and contains many transcendental numbers such as $\pi$, as well as some interesting classes of numbers e.g.~the {\em multiple zeta values} (see~\cite{Wald00}).

Heuristically, the main point in the {\em philosophy of periods} is the study of transcendental numbers and their relations via integration of algebro-geometric objects with coefficients in $\RRalg$ under a set of transformations rules. One of the main difficulties to deal with such transformations rules is to ensure that they preserve the $\RRalg$ character of the underlying algebro-geometric objects.
In the following, any (semi-)algebraic object is assumed to be defined with coefficients in $\RRalg$.\\

The main conjecture for periods is called the \emph{Kontsevich-Zagier period conjecture} (also \emph{KZ-conjecture}, for brevity's sake), originally stated as follows.

\begin{conj*}[{\cite[Conjecture 1]{KonZag01}}]\label{conj:KZ}
    If a real period admits two integral representations, then we can pass from one formulation to the other using only three operations: integral additions by domains or integrands, change of variables and the Stokes formula. Moreover, these operations should respect the class of objects previously defined.
\end{conj*}

The above three operations are called the {\em KZ-rules}. A common idea around the Kontsevich-Zagier period conjecture is that, in their own words, \emph{``this problem looks completely intractable and may remain so for many years''} (see also~\cite{Wald06, Andre12, Ayoub15}). This is due to different facts, e.g.:
\begin{itemize}
	\item Although this conjecture is natural, its place among classical conjectures in Number Theory is still not well understood.
	This is not the case of the Diophantine conjecture, which is in the same spirit as the KZ-conjecture and deals with relations between multiple zeta values, see~\cite[p. 587]{Wald00}.
	
	\item Up to our knowledge, there is no strategy of proof for the KZ-conjecture.
\end{itemize} 
Nevertheless, there exists a motivic version of the KZ-conjecture, which is better understood and related to the original Grothendieck period conjecture, see~\cite{Andre12}. The case of periods of {\em 1-motives} is succinctly studied by Huber and W\"ustholz in~\cite{HuberWust22:1periods}, proving the motivic conjecture in this context.\\

The aim of this paper is to discuss possible reformulations of this conjecture which suggest a potential scheme of proof. Also, we give some hints about the role that each of the three KZ-rules plays in the conjecture and related problems. The main ingredient is the following algorithmic result by the second author:

\begin{thm*}[Semi-canonical reduction, {\cite[Thm. 1.1]{Viu15}}]\label{thm:semican_form_thm}
  Let $p$ be a non-zero real period given in a certain integral form $\Ical(S,P/Q)$ in $\RR^d$ as in (\ref{eqn:int_form}). There exists an effective algorithm only using the KZ-rules such that $\Ical(S,P/Q)$ can be rewritten as
  \[
     \Ical(S,P/Q) = \sgn(p)\cdot\vol_m(K),
  \]
  where $K\subset\RR^m$ is a compact top-dimensional semi-algebraic set and $\vol_m(\cdot)$ is the canonical volume in $\RR^m$, for some $0<m\leq d+1$.
\end{thm*}

As a consequence, any real period can be written up to sign as the volume of a compact semi-algebraic set. The above theorem can be extended for the whole set of periods $\Pkz\subset\CC$ considering representations of the real and imaginary part respectively. Such a representation for a period $p$ is called \emph{a geometric semi-canonical representation of $p$}.\\

The semi-canonical reduction suggests a more geometrical point of view in order to deal with relations between period representations. Indeed, let us consider the equivalence relation
\begin{equation}
\label{equa1}
\mathcal{I} (S,\omega ) \underset{KZ}{\sim}  \mathcal{I} (\tilde{S},\tilde{\omega} ) ,
\end{equation}
defined between two representations of the same period which can be related by a finite sequence of KZ-rules. By the semi-canonical reduction theorem above, we have
\begin{equation*}
\mathcal{I} (S,\omega ) \underset{KZ}{\sim}  \mathcal{I} (K,1)
\ \ \mbox{\rm and}\ \ 
\mathcal{I} (\tilde{S},\tilde{\omega} ) \underset{KZ}{\sim}  \mathcal{I} (\tilde{K},1) ,
\end{equation*}
so that the equivalence relation (\ref{equa1}) implies that 
\begin{equation*}
\mathcal{I} (K,1 ) \underset{KZ}{\sim}  \mathcal{I} (\tilde{K},1 ) .
\end{equation*}
As a consequence, the Kontsevich-Zagier conjecture is equivalent to :

\begin{conj*} Let $K_1$ and $K_2$ be two compact top-dimensional semi-algebraic sets in $\RR^d$ such that $\vol_d(K_1)=\vol_d(K_2)$. Then we can pass from one formulation to the other one by only using transformations respecting the KZ-rules.
\end{conj*}

This conjecture looks like a kind of {\em Hilbert's third problem} for semi-canonical representations of periods.
However, it is worth noticing that the class of transformations respecting the KZ-rules do not preserve semi-canonical representations of periods as volumes of compact top-dimensional semi-algebraic sets.\\

The previous remark suggests that we should adapt the KZ-rules into our geometrical picture. Taking into account the classical formulation of the third Hilbert's problem~\cite{Cartier86}, we propose the following {\em geometric Kontsevitch-Zagier's conjecture} (or {\em GKZ-conjecture}).

\begin{conj}[GKZ-conjecture]\label{conj:geometric_KZ}
Let $K_1, K_2$ be two compact top-dimensional semi-algebraic sets in $\RR^d$ such that $\vol_d(K_1)=\vol_d(K_2)$. Then, we can transform $K_1$ into $K_2$ only using the following geometric operations:
\begin{itemize}
\item Semi-algebraic scissors congruences, 

\item Algebraic volume-preserving maps respecting the KZ-rules,

\item (Cartesian) product relations.
\end{itemize}
Equivalently, let $\Kcal_0(\Ccal\saq)$ be the Grothendieck ring of top-dimensional compact semi-algebraic sets modulo decompositions and volume-preserving transformations. Then, one has that $[K_1]$ and $[K_2]$ are equal in $\Kcal_0(\Ccal\saq)$, i.e. the morphism $\vol:\Kcal_0(\Ccal\saq)\to\Pkz^\RR$ (associating the class $[K]$ to the volume of $K$) is injective.
\end{conj}

More details on these operations and the equivalent reformulation as a {\em cut-and-paste} problem in a certain Grothendieck ring will be given in Section \ref{sec:reformulation}.\\

The geometric operations involved in the GKZ-conjecture, called the {\em GKZ-rules} in the following, are compatible by definition with the KZ-rules. As a consequence, we have the following relation between the two conjectures:
\begin{equation}
\mathrm{GKZ} \Longrightarrow \mathrm{KZ} .
\end{equation}

Related to the above problem, we state another one based on \emph{piecewise linear} (\PL) geometry, called {\em \PL-geometric Kontsevitch-Zagier's conjecture} (or {\em \PL-GKZ-conjecture}), as follows:

\begin{conj}[\PL-GKZ-conjecture]
\label{conj:PL_KZ}
    Let $P_1$ and $P_2$ be two rational polyhedra in $\RR^d$ equipped respectively with two piecewise algebraic volume forms $\omega_1$ and $\omega_2$ such that $\int_{P_1}\omega_1=\int_{P_2}\omega_2$. Then, we can pass from one integral to the other one only by rational polyhedra scissors-congruences and rational polyhedra transformations.
\end{conj}

The plan of the paper is as follows: we introduce in Section~\ref{sec:reformulation} a first discussion about the nature of Stokes's formula in the original formulation of the Kontsevich-Zagier period conjecture and its possible replacement by the \emph{Fubini} and \emph{Newton-Leibniz} formulas as allowed operations. Then, following the spirit of Kontsevich-Zagier's period conjecture, we present a geometric problem dealing about the smaller class of operations which relates compact semi-algebraic sets of equal volume
(the GKZ-conjecture).

Based on the fact that any compact semi-algebraic set is (semi-algebraically) triangulable, we formulate in Section~\ref{sec:SAStoPL} an analogue of the GKZ-conjecture in \PL-geometry: the generalized Hilbert's third problem for rational polyhedra
(the \PL-GKZ-conjecture).
In this case, the rules of transformation between rational polyhedra have more combinatorial nature. Some known (partial) results on the generalized Hilbert third problem are discussed in Section~\ref{sec:about}, specially in the case of different polyhedra with equal volume with respect to the canonical volume form. In particular, we emphasize here the case of convex polyhedra.

Finally, we discuss in Section~\ref{sec:KZconj_conclu} how the different problems, results and obstructions presented in this article outline a first scheme for either proving the geometric Kontsevich-Zagier conjecture or finding a potential counterexample.

\section{A reformulation of the Kontsevich-Zagier period conjecture}\label{sec:reformulation}

Our aim is to present a modified form of the Kontsevich-Zagier period conjecture from the point of view of volume-preserving operations between compact semi-algebraic sets. We are motivated by the following:
\begin{itemize}
    \item Any operation should be simple and natural from the point of view of integral calculus.
    
    \item Period representations and its operations should be easy to handle and implement.
\end{itemize}

\subsection{Discussion about Stokes' formula versus Fubini's theorem}

Sums by integrand functions or domains and change of variables are natural and induce explicit formulas between integrals. On the other hand, computations using Stokes' formula require an exhaustive analysis in order to determine partitions and (semi-algebraic) parameterizations of the boundary.  Our main concern is then to find an alternative to the use of Stokes' formula, which leads to more tractable manipulations.\\

How such an alternative to Stokes' formula could be determined? An idea is to recover basic operations of integral calculus satisfying our previous constraints and which suffice to prove the Stokes formula. Classically the \emph{Fubini theorem} as well as the \emph{Newton-Leibniz formula} (i.e. \emph{Fundamental theorem of calculus}) are important ingredients of the proof~\cite[p.~253-254]{Spivak79bookI}.\\

Thus, a first natural tentative would be to replace the Stokes formula by Fubini's theorem and the Newton-Leibniz formula. This choice is also motivated by geometric considerations related to the semi-canonical reduction for periods. Indeed, Fubini's theorem is a convenient tool to bring down the dimension of the semi-canonical reduction and this plays an important role in the transcendence of periods, as it is pointed out in Remark~\ref{rem:final_GKZ}. However, this strategy can not be achieved for at least two reasons. First, primitives of algebraic functions are in general transcendental functions. As a consequence, the integral representations obtained by Fubini's theorem are out of the algebraic class. This is illustrated in the following expression:
\begin{equation}\label{eqn:fubini}
    \int_{P\times Q} f(x,y)\ \dint x\dint y=\int_{P}\left[\int_{Q} f(x,y)\ \dint x\right]\dint y
\end{equation}
For example:
\begin{align*}
    \int_{\left\lbrace\substack{0<x<1\\ 1<y<x+1}\right\rbrace} \frac{x}{y}\ \dint x\dint y & = \int_0^1\left(\int_1^{1+x} \frac{x}{y} \ \dint y\right)\dint x\\
							  &=\int_0^1 x\log(x+1)\ \dint x\\
							  &=\frac{1}{4}.
\end{align*}

\begin{rem}
    In~\cite[p.~5]{KonZag01}, Kontsevich and Zagier affirm that an integral of a transcendental function could be a period \emph{``by accident''}, giving as example the following expression:
    \[
        \int_0^1 \frac{x}{\log\frac{1}{1-x}}\ \dint x=\log 2.
    \]
    However, the previous discussion shows that this is not an ``accident'' at all: one can express many periods as integrals of transcendental functions by using Fubini's theorem.
\end{rem}

Secondly, in the case where Fubini's theorem gives an integral in the algebraic class, e.g.~when one takes $f(x,y)=g(x)h(y)$ in~\eqref{eqn:fubini}, this induces potential quadratic relations between periods.

\subsection{A geometric Kontsevich-Zagier problem for periods}\label{subsec:geom_KZ}

Our idea is to take advantage of the geometric representation obtained in the semi-canonical reduction theorem in order to formulate a related geometric problem in terms of volumes of compact semi-algebraic sets. Also, we set a class of allowed operations respecting representation of periods as volumes of compact semi-algebraic sets.\\

Let $\Ccal\saq^d$\label{csaq} be the collection of compact $d$-dimensional semi-algebraic sets in $\RR^d$, containing the empty set and closed under finite unions. %
We define:
\begin{equation}\label{eqn:CAS}
 \Ccal\saq=\bigcup_{d\geq 1}\Ccal\saq^d
\end{equation}
We study the different relations in $\Ccal\saq$ which are possible to obtain with the geometric operations described as follows. Let $(\Kcal_0(\Ccal\saq), +)$ be the free abelian group generated by the classes $[K]$, with $K\in\Ccal\saq$, modulo the relations:
\begin{itemize}
    \item {(\bf Semi-algebraic scissors congruences)} If $K,K'\in\Ccal\saq^d$ such that $\codim(K\cap K')\geq1$, then
    \begin{equation}\label{eqn:SA-scissors}
        [K\cup K'] = [K] + [K'].
    \end{equation}
    
    \item {(\bf Algebraic volume-preserving maps)} Let $K\in\Kcal_0(\Ccal\saq)$. If there exist $U,V\subset\RR^d$ open semi-algebraic sets and $f:U\to V$ an algebraic\footnote{We say that $f$ is \emph{(semi-)algebraic} if its graph $\Gamma_f$ is a (semi-)algebraic set in $U\times V$.} diffeomorphism such that $K\subset U$ and $\det(\Jac(f))=1$, then
    \begin{equation}\label{eqn:SA-volpresmaps}
        [K] = [f(K)].
    \end{equation}
    We denote by $K\simeq_{\vol} K'$ when $K'=f(K)$ by such a volume-preserving map $f$.
\end{itemize}
A natural connection between compact semi-algebraic sets in $\Ccal\saq$ comes from the direct product and Fubini's theorem. We call the following relations that induce a ring structure in $\Kcal_0(\Ccal\saq)$ the \emph{(Cartesian) product relations}:

\begin{itemize}
	\item {(\bf Product relation)} For any $K\in\Ccal\saq^d$ and $K'\in\Ccal\saq^{d'}$, then
	\begin{equation}\label{eqn:SA-product}
		[K\times K']= [K]\cdot [K'].
	\end{equation}

	\item {(\bf Flattening relation)} Let $I=[0,1]$ be the unit interval. For any  $K\in\Ccal\saq^d$ and $n\geq0$, then
		\begin{equation}\label{eqn:SA-flattening}
			[K\times I^n]= [K].
		\end{equation}
\end{itemize}

\begin{rem}\mbox{}
	\begin{enumerate}
	\item The class of $I$ is the neutral element of the product in $\Kcal_0(\Ccal\saq)$ and we denote $1=[I]$. Analogously, the class of the empty set is the neutral element of the sum, denoted $0=[\emptyset]$.
	
	\item It is worth noticing that the point $\RR^0=\{\text{pt}\}$ is not considered as an element of $\Ccal\saq$ in (\ref{eqn:CAS}). Beyond the fact that points have no intrinsic euclidean volume, one can see that any class of a point would define a zero divisor
	since $\{\text{pt}\}\times I\subset\RR^2$ has codimension $1$.
	\end{enumerate}
\end{rem}

\begin{ex}
	Let $a,b\in\RRalg$ be positive. Define the compact semi-algebraic set
	\[
		K_{a,b}=\{(x,y)\in\RR^2 \mid 0\leq x\leq a, 0\leq y+\sqrt{x+1}\leq b\}
	\]
	Then, $K_{a,b}\simeq_{\vol} C_{ab}$, where $C_{ab}=[0,ab]\times I$ is the square with sides of lengths $ab$ and $1$ respectively, via the composition of a linear isometric map of $\RRalg^2$ and the algebraic map $y\mapsto y-\sqrt{x+1}$ defined in the positive semi-plane. In particular,
	\[
		[K_{a,b}]=[C_{ab}]=[I_{ab}],
	\]
	via the flattening relation, where $I_{ab}$ is the closed interval $[0,ab]$ in $\RR^2$.
\end{ex}

The product on $\Kcal_0(\Ccal\saq)$ extends to a particular class of fibrations. Let $E,B,F\in\Ccal\saq$ and $\rho:E\to B$ be a {\em proper} semi-algebraic map, i.e. verifying that $\rho^{-1}(C)\in\Ccal\saq$ for any subset $C$ of $B$ such that $C\in\Ccal\saq$. We say that $\rho:E\to B$ is a {\em piecewise trivial proper fibration with fiber $F$} if there exists a finite partition $B=\bigcup_i C_i$ such that $C_i\in\Ccal\saq$, $\codim(C_i\cap C_j)\geq 1$ and $\rho^{-1}(C_i)\simeq_{\vol} C_i\times F$, for any $i,j$.

We have the following straightforward property.

\begin{propo}
	If $\rho:E\to B$ is a piecewise trivial proper fibration with fiber $F$, then $[E]=[F]\cdot [B]$.
\end{propo}

We extend naturally the volume of semi-algebraic sets into a well-defined ring morphism $\vol:\Kcal_0(\Ccal\saq)\to\RR$ given by $\vol([K])=\vol_d(K)$, where $d$ is the unique natural number such that $K\in\Ccal\saq^d$. By the semi-canonical reduction theorem, 
the previous map induces a surjective morphism
\[ 
\begin{tikzcd}
	\Kcal_0(\Ccal\saq) \arrow{r}{\vol} & \Pkz^\RR.
\end{tikzcd}
\]

In conclusion, the previous relations defining $\Kcal_0(\Ccal\saq)$ can be naturally seen as geometric operations between compact top-dimensional semi-algebraic sets with equal volume and one could ask whether these operations are sufficient to relate them. This is given by the injectivity of the morphism $\vol$, as it is stated in the GKZ-conjecture (Conjecture~\ref{conj:geometric_KZ}) in the Introduction.\\

From the previous discussions and the fact that any integral representation of a non-zero period can be algorithmically reduced to the volume of a compact semi-algebraic set (up to sign), one has the following implication.

\begin{thm}
  The GKZ-conjecture implies the KZ-conjecture.
\end{thm}
A negative answer for the GKZ-conjecture will be also very interesting in order to determine possible obstructions or counterexamples for the KZ-conjecture.\\

The above equivalent geometrical problems posed by the GKZ-conjecture are reminiscent of two classical problems about \emph{cutting-and-pasting} in geometry: Hilbert's third problem (see~\cite{Cartier86}) and Tarski's circle-squaring problem (see~\cite{Laczkovich90}).

Hilbert's third problem deals with how scissor-congruence operations relate $n$-polyhedra of the same volume. The 3-dimensional case was solved in 1900 by Dehn~\cite{Dehn01, Sydler65} who introduced the so-called \emph{Dehn invariant} to distinguish scissor-congruent polyhedra. The main difference with our problem in terms of \emph{cutting-and-pasting} operations is that we work in the semi-algebraic class (with coefficients in $\RRalg$). This is a weaker scissor-congruence constrain than in the original Hilbert problem, where one must stay in the polyhedral class. Also, the class of volume-preserving transformations allowed for each piece is larger than classical isometries on the affine space.

The Tarski circle-squaring problem poses an equivalent question in the context of scissor-congruences between a square and a $2$-disk with equal areas, but without restrictions on the class of possible decompositions. This problem was solved by M.~Laczkovich~\cite{Laczkovich90} in 1990, proving that there exists a decomposition of the circle by non-measurable sets covering the square only using translations. This problem is treated in~\cite{CVS16_2} in the context of our setting, were the semi-algebraic class has strong measurable properties and we face constraints of arithmetic nature due to the choice of coefficients in $\RRalg$.\\

Here, we present our problem in a more modern language, where the construction of $\Kcal_0(\Ccal\saq)$ encodes the geometrical relations in the same way as in algebraic K-theory and Grothendieck rings theory (e.g.~\cite[p.~268]{Cartier86} and~\cite{Zakharevich16}).

There exist a recent cutting-and-pasting problem in the setting of the Grothendieck ring of complex varieties $\Kcal_0(\text{Var}_\CC)$, which is a ring generated by classes $[X]$ of complex varieties modulo isomorphisms, subtractions $[X\setminus F]=[X]-[F]$ by closed subsets $F\subset X$ and the product relation. This problem is known as the \emph{Larsen-Lunts Conjecture}: an equality of classes $[X]=[Y]$ implies that  $X$ and $Y$ admit a decomposition into isomorphic locally closed subvarieties, see~\cite{LarsenLunts14}. This was proved to be true when $\dim X\leq1$ in~\cite{LiuSebag10}, but it turns out to be false in general as a consequence of the existence of zero divisors in $\Kcal_0(\text{Var}_\CC)$, see e.g.~\cite{Borisov18}.

\begin{rem}\label{rem:final_GKZ}~
	\begin{enumerate}
		\item Following Cartier's presentation of Dehn's invariant, it would be interesting to construct a similar invariant for $\Kcal_0(\Ccal\saq)$. A natural question that arises is whether the ring $\Kcal_0(\Ccal\saq)$ is a domain, i.e. if it contains zero divisors as $\Kcal_0(\text{Var}_\CC)$ does, even if intuitively this should not be the case.
		
		\item The study of the \emph{flattering relation} and trivialization by fibrations with fiber $I$ is central in the characterization of ``minimal'' representations of compact semi-algebraic sets in $\Kcal_0(\Ccal\saq)$, i.e. given $[K]$ then to find the minimal $d_0\geq1$ such that one can express $[K]=[\widetilde{K}]$ for some $\widetilde{K}\in\Ccal\saq^{d_0}$. This problem will play a fundamental role in future works concerning \emph{degree theory and transcendence of periods}, see~\cite{CVS16_2}.

		\item It is worth noticing that the general problem of classifying $\RR$-semi-algebraic sets up to $\RR$-semi-algebraic volume-preserving bijections modulo lower-dimensional sets was already considered by A.~Blass and S.~Schanuel in~\cite[p.~3]{BlassSchanuel} as an intractable problem.
	\end{enumerate}	
\end{rem}

\subsection{Comparison with the Ayoub approach}

In~\cite[Conj.~1.1]{Ayoub15}, J.~Ayoub considers
a different approach of the KZ-conjecture by concentrating all the complexity of the period representation on the differential form instead of on the domain. More precisely, he identifies periods as complex integrals over a real unit hypercube as follows:
\begin{equation}\label{eqn:Ayoub}
	\int_{[0,1]^d}{f(z_1,\ldots,z_d)}\dint z_1\cdots\dint z_d,
\end{equation}
where $f$ is a holomorphic function over the closed poly-disk $\mathbb{D}^d=\{|z_i|\leq1\}\subset\CC^d$ and which is algebraic over $\RRalg(z_1,\ldots,z_d)$, see~\cite[\S2.2]{Ayoub14}. Using this representation, he formulates a \emph{compact form} of the Kontsevitch-Zagier conjecture: any relation between periods written as in (\ref{eqn:Ayoub}) is a consequence of linear combinations of expressions as follows:
\begin{equation}\label{eqn:Ayoub_rels}
	\frac{\partial f}{\partial z_i}-f_{|z_i=1}+f_{|z_i=0}.
\end{equation}
Note that one recovers the Stokes formula from the above. Moreover, these expressions are supposed to contain the change of variables relation~\cite[Rem.~1.5]{Ayoub15}. \\

Following our construction of the geometric Kontsevich-Zagier conjecture, we can recast Ayoub's construction as follows: any period $p=\mathcal{I} (S_1,\omega_1 ) + i\cdot\mathcal{I} (S_2,\omega_2 )\in\Pkz$ can be written as the value of a complex integral $\mathcal{I} ([0,1]^n ,\eta_{|\mathbb{D}^d} )$, where $\eta_{|\mathbb{D}^d}$ is a top-dimensional complex holomorphic form over $\mathbb{D}^d$. As for the GKZ-conjecture, the KZ-rules are adapted to this new setting by specifying Ayoub's rules (or \emph{AKZ-rules}) as in (\ref{eqn:Ayoub_rels}).

Nevertheless, at this moment it is not known by the authors if there exists an algorithmic way to pass from an expression (\ref{eqn:int_form}) to another of type (\ref{eqn:Ayoub}) corresponding to the same effective period, as one has for the semi-canonical representation.

A \emph{relative version} of the Kontsevich-Zagier conjecture is proved by the same author~\cite[Thm.~4.25]{Ayoub15}, exhibiting explicitly the ``geometric'' relations between Laurent series where periods appear as coefficients of the form (\ref{eqn:Ayoub}). Roughly speaking, these relations are linearly generated by expressions of type (\ref{eqn:Ayoub_rels}) adapted for such Laurent series
and by functions which have zero-valued integrals (\ref{eqn:Ayoub}).\\

Both our approach of the KZ-conjecture as well as the one proposed by Ayoub deal with the following problem: to describe equivalence classes of volume elements on a given manifold under the action of volume-preserving maps, as studied by Moser in~\cite{Moser65}.

It is worth noticing that one of the results in Moser's article could potentially give a solution of the conjecture following Ayoub's point of view: it is proved in~\cite[Lemma 2]{Moser65} that there exists a change of variables relating two integrals of differentiable functions over the unit hypercube with equal total volume. Moreover, such a change of variables is explicitly constructed in the proof, but it turns to be \emph{transcendental} in most of the cases, since the construction is made by using integration inductively over each variable.\\

\section{From Semi-algebraic to Piecewise Linear geometry}\label{sec:SAStoPL}

We consider an problem in the \emph{Piecewise Linear} (\PL) class analogous to the previous one. The interest of this new problem is twofold. First, we can make a connection between the \PL~and semi-algebraic cases. Second, we can take advantage of known results in \PL-geometry related to our problem.

Up to now we have not considered all the good properties of semi-algebraic sets. In particular, it is well-known that compact semi-algebraic sets admit ``good'' triangulations, which allows us to obtain a representation in the \PL~category. Using this result, called \emph{Semi-algebraic Triangulation Theorem}, we obtain a ``\PL~analogue'' of the the geometric Kontsevich-Zagier problem.

\subsection{Semi-algebraic triangulations}

Let $a_0,\ldots,a_k$ be $k+1$ points affinely independent in $\RR^d$, for $k\geq0$. Recall that a \emph{$k$-simplex} $[a_0,\ldots,a_k]$ is the set of points $p\in\RR^d$ such that there exist non-negative $\lambda_0,\ldots,\lambda_k\in\RR$ verifying that $\sum_{i=1}^k\lambda_i=0$ and $p=\sum_{i=0}^k\lambda_ia_i$. In this case, the numbers $(\lambda_0,\ldots,\lambda_k)$ are called the \emph{barycentric coordinates} of $p$. For any non-empty subset $\{a_{i_0},\ldots,a_{i_\ell}\}\subset\{a_0,\ldots,a_k\}$, the $\ell$-simplex $[a_{i_0},\ldots,a_{i_\ell}]$ is called a $\ell$-\emph{face} of $[a_0,\ldots,a_k]$. If $\sigma$ is a simplex, we denote by $\mathring{\sigma}$ the \emph{open simplex}, composed by the points of $\sigma$ whose barycentric coordinates are all positive.

A \emph{finite simplicial complex} of $\RR^d$ is a collection of simplices $\Kcal=(\sigma_i)_{i=1,\ldots,m}$ verifying that the faces of every $\sigma_i$ belong to $\Kcal$ and such that $\sigma_i\cap\sigma_j$ is either empty or a common face of $\sigma_i$ and $\sigma_j$, for every $1\leq i<j\leq m$. The \emph{realization} of the complex is $|\Kcal|=\bigcup_{i=1}^m\sigma_i$. Note that the open simplices $\mathring{\sigma}_i$ form a partition of $|\Kcal|$.

Following e.g.~\cite[Sec. 9.2]{BCR98}, one has the following.
\begin{thm}[Semi-algebraic Triangulation Theorem]
    Every compact semi-algebraic set $K$ in $\RR^d$ is \emph{semi-algebraically triangulable}, i.e. there exists a finite simplicial complex $\Kcal=(\sigma_i)_{i=1,\ldots,p}$ and a semi-algebraic homeomorphism $\Phi: |\Kcal|\to K$.
\end{thm}

It is important to notice that this triangulation procedure is algorithmic (see~\cite[Sec.~5.7, p.~183]{BPR06}) and moreover that any operation used in this procedure belongs to the KZ-rules. In particular, every stratum of the decomposition is obtained by a {\em Nash embedding}~\cite[Remark 9.2.5]{BPR06}. In addition, Ohmoto and Shiota proved in~\cite{OhmotoShiota15} that this triangulation can be chosen such that the semi-algebraic homeomorphism $\Phi: |\Kcal|\to K$ is of class $C^{1}$, i.e. it has an extension to a $C^{1}$-mapping defined on an open semi-algebraic neighborhood of $|\Kcal|$ in $\RR^d$. However this construction is not defined algorithmically {\em a priori}.

\begin{lem}\label{lem:SAtoPL}
    For any $K\in\Ccal\saq^d$, there exist a rational polyhedron $P$ in $\RR^d$ and a piecewise algebraic volume form%
    \footnote{%
    This notion was introduced in Kontsevich-Soibelman's construction of the de Rham homotopy theory of semialgebraic sets~\cite{KonSoi00,HLTV11}, or more precisely the study of {\em PA forms} on {\em PA spaces}. In our setting, this means a continuous semi-algebraic differential form.
    } %
    $\omega$ such that
    \[
        \vol_d(K)=\int_P \omega.
    \]
    Moreover, the passage between the two integrals respects the KZ-rules.
\end{lem}

\begin{proof}
    It follows from~\cite[Thm~1.1]{OhmotoShiota15} that there exists a semi-algebraic triangulation $(\Kcal,\Phi)$ of $K$ such that $\Phi: |\Kcal|\to K$ is of class $C^{1}$. Then, the pull-back $\Phi^*$ is well defined over $K$ and taking $P=|\Kcal|$ we have:
    \[
	    \vol_d(K)=\int_K \omega_0=\int_P \Phi^*(\omega_0),
    \]
    where $\omega_0$ is the canonical volume form in $\RR^d$. The form $\omega=\Phi^*(\omega_0)$ can be expressed as
    $\dint f_1\wedge\cdots\wedge\dint f_d$ for some $\Ccal^1$ semi-algebraic maps $\{f_i\}_{i=1}^d$. Since locally we have only performed piecewise algebraic change of variables over semi-algebraic decompositions of $K$, the result follows.
\end{proof}

\subsection{A \PL~version of the geometric Kontsevich-Zagier problem}

The previous lemma leads us to formulate an analogous version of the geometric Kontsevich-Zagier problem in terms of \PL-geometry, the \PL-GKZ-conjecture (Conjecture~\ref{conj:PL_KZ}), presented in the Introduction. In the following, we call {\em rational polyhedra transformation} to any \PL-map between two rational polyhedra.

Due to the similarity between
the GKZ-conjecture and the \PL-GKZ-conjecture, it seems that we have not gained any advantage from the new formulation. Moreover, it is worth noticing that the (cartesian) product relation~\eqref{eqn:SA-product} does not appear naturally in this setting and we are fixing the dimension of our polyhedra.

However, both rational polyhedra scissors-congruence and rational polyhedra transformations induce an equivalent GKZ-rule in the compact semi-algebraic setting.

\begin{propo}\label{propo:PLtoSA}
    Let $K$ and $K'$ be top-dimensional compact semi-algebraic sets in $\RR^d$ and consider $(P,\omega)$ and $(P',\omega')$ two pairs obtained by applying Lemma~\ref{lem:SAtoPL} on $K$ and $K'$, respectively. Then:
    \begin{enumerate}
      \item Any rational polyhedra scissors-congruence operation on $P$ induces a semi-algebraic decomposition on $K$.
      \item Any rational polyhedra transformation between $(P,\omega)$ and $(P',\omega')$ induces a semi-algebraic volume-preserving homeomorphism between $K$ and $K'$.
    \end{enumerate}
\end{propo}

\begin{proof}
  Consider the semi-algebraic triangulation $\Phi: P\to K$ of class $C^{1}$. We can express $w=\dint f_1\wedge\cdots\wedge\dint f_d$ for some $\Ccal^1$ semi-algebraic maps $\{f_i\}_{i=1}^d$ such that $\Phi=(f_1,\ldots,f_d)$.\\%
  Case (1) is straightforward: without loss of generality, assume that the scissors-congruence operation on $P$ produces two top-dimesional polyhedra $P_1$ and $P_2$ both equipped with $\omega$ as differential form. Consider $K_1$ and $K_2$ as the images of $\Phi_{|P_1}$ and $\Phi_{|P_2}$, respectively. Then $K$ decomposes semi-algebraically on $K_1$ and $K_2$ verifying that $\vol_d (K) = \vol_d(K_1) + \vol_d(K_2)$, since the push-forward $\Phi_*(\omega)$ is the canonical volume form in $\RR^d$.\\
  For the case (2), consider $L:P\to P'$ a \PL-isomorphism such that $L_*(\omega)=\omega'$. Also, let $\Phi': P'\to K'$ be a semi-algebraic triangulation of class $C^{1}$ of $K'$ such that $\omega'$ is the pull-back of the canonical volume form by $\Phi'$. Defining the semi-algebraic map $\Psi=\Phi'\circ L\circ \Phi^{-1}$, one has that $K'=(\Phi'\circ L)(P)=\Psi(K)$. Also, by construction:
  \[
    \vol_d (K) = \int_P \omega = \int_{P'} \omega'=\vol_d(K').
  \]
\end{proof}

\begin{thm}
  If there exists $d_0\geq1$ such that for any $d\geq d_0$ the \PL-GKZ-conjecture for $d$-dimensional polyhedra holds, then the KZ-conjecture is true.
\end{thm}

\begin{proof}
  Let $\Ical(S_1,\omega_1)$ and $\Ical(S_2,\omega_2)$ be two integral representations of the same period $p$ and consider $K_1$ and $K_2$ the respective semi-canonical reductions associated to them. Taking $m=\max\{\dim K_1, \dim K_2, d_0\}$, there exist natural numbers $n_1,n_2\geq0$ such that $m=\dim K_1 + n_1 = \dim K_2 + n_2$. Considering the following compact semi-algebraic sets:
  \[
    \tilde{K}_1 = K_1\times I^{n_1}\quad\text{ and }\quad \tilde{K}_2 = K_2\times I^{n_2},
  \]
  it is straightforward that $\vol_m (\tilde{K}_1)=\vol_m (\tilde{K}_2) =p$. Using Lemma~\ref{lem:SAtoPL}, we obtain two pairs $(P_1,\omega^{PA}_1)$ and $(P_2,\omega^{PA}_2)$ verifying the hypotheses of the \PL-GKZ-conjecture for $m$-dimensional polyhedra. By hypothesis, one can pass from $\int_{P_1}\omega_1^{PA}$ to $\int_{P_2}\omega_2^{PA}$ only by a finite sequence of rational polyhedra scissors-congruence and rational polyhedra transformations. By Proposition~\ref{propo:PLtoSA}, this induces a finite sequence of KZ-rules from $K_1$ to $K_2$. Thus, the KZ-conjecture holds.

\end{proof}

The above geometric problem about periods belongs now to the study of discrete and polyhedral geometry, for which many powerful algorithmic methods already exist to deal with scissor-congruences. We refer to the book of I.~Pak~\cite{Pak15} for an overview on the state of art of discrete and polyhedral geometry together with combinatorial methods.

In the next section, we give some hints about why the previous problem could have a positive answer. However, we are still far away from constructing a complete proof, and this is related to the following fact: we are using some classical results for which the current known proof is unsatisfying in our setting.\\

We exhibit again in this other approach how essential is the decomposition-on-domains (or the scissors-congruence) operation in the different conjectures presented in this article. If one removes the scissors-congruence operation, then the main idea to find a counterexample to the \PL-GKZ-conjecture is to construct two \PL-manifolds in the arithmetic class which are homeomorphic but not \PL-homeomorphic. %
However, as a straightforward consequence of the {\em Hauptvermutung} for compact semi-algebraic sets~\cite[Cor~4.3]{ShiotaYokoi84} proved by M.~Shiota and M.~Yokoi, one has the following obstruction.

\begin{propo}
  If a period admits two representations by polyhedra $P_1$ and $P_2$ as in %
  the \PL-GKZ-conjecture %
  such that $P_1$ is homeomorphic to $P_2$ but not \PL-homeomorphic, then there is no global (volume-preserving) semi-algebraic mapping between $P_1$ and $P_2$.
\end{propo}

The above result is due to the fact that triangulations of compact semi-algebraic sets are unique up to \PL-homeomorphism. As a consequence, there is no global volume-preserving \PL-mapping between the above polyhedra. This follows also from  Milnor, Kirby and Siebenmann's example~\cite{Milnor61,KirSieb69}.

\section{About volumes of rational polyhedra, scissor-congruences and mappings}\label{sec:about}

In this section, we exhibit some known partial results about the problem of constructing transformations between two polyhedra equipped with volume forms which have the same total volume.

\subsection{Canonical volume form}
Before discussing the general case, a first simplification of
the \PL-GKZ-conjecture
is to consider the polyhedra equipped with the canonical volume form.

\subsubsection{A general result}
In that case, we have found an interesting result discussed by A.~Henriques and I.~Pak in~\cite{HenriquesPak04} which gives a priori a positive answer to the problem. The strategy is first to look at this problem restricted to convex polytopes and then to reduce the general case using appropriate decompositions by convex parts.

\begin{thm}[{\cite[Example~18.3, p.~171]{Pak15}}]\label{thm:pak15}
    Let $P,Q\subset\RR^d$ be two rational polyhedra of equal volume, i.e. $\vol_d(P)=\vol_d(Q)$. Then one can pass from $P$ to $Q$ only using rational polyhedra scissors-congruences and rational polyhedra transformations.
\end{thm}

It is worth noticing that this result only deals with periods coming from volumes of rational polyhedra, i.e. $\QQ\oplus i\QQ\subset\Pkz$. Nevertheless, the main difference here is that we allow an operation which does not appear in the original Hilbert third problem: rational polyhedra transformations. The importance of allowing such type of operations is clear, since the original Dehn invariant becomes an obstruction otherwise.

The main ingredient of the proof given for the latter is the following result:

\begin{thm}[{\cite[Theorem~2]{HenriquesPak04}}]\label{thm:henripak04}
    Let $P,Q\subset\RR^d$ be two convex rational polyhedra of equal volume, i.e. $\vol_d(P)=\vol_d(Q)$. Then there exists a one-to-one rational map $f:P\to Q$, which is continuous, piecewise linear and volume-preserving.
\end{thm}

Note that no decomposition of the polyhedra is used in the previous result, due to the restrictive hypothesis in this setting. 
However, the proof of the latter result relies fundamentally on \emph{Moser's theorem}~\cite{Moser65}
about volume-preserving maps for manifolds, which is not constructive. In fact, to obtain a combinatorial proof of Moser's Theorem for convex polyhedra is still an open problem, see~\cite[Sec.~8.7, p.~17]{HenriquesPak04}. Pak refers to continuous, piecewise linear and volume-preserving maps as \emph{Monge maps} (see~\cite[Sec.~18.1, p.~170]{Pak15}).

We give an illustration of Theorem~\ref{thm:pak15} for the case of convex plane polygons in the following. The main point is that in this case the Monge maps are explicitly constructed and one sees that these maps are rational. We then discuss how Theorem~\ref{thm:pak15} (non-convex case) can be derived from Theorem~\ref{thm:henripak04} (convex case).

\subsubsection{Convex polygons}

Restricting to polygons in the real plane, Pak describes a simple way to obtain a Monge map between two convex polygons of same area~\cite[Ex.~18.2, p.~170]{Pak15}. The idea is to reduce any convex polygon into a triangle, by a recursive series of continuous, piecewise linear and volume-preserving transformations displacing vertex in order to reduce the number of faces of the polygon.

Let $P$ be a polygon of $n> 3$ faces and consider a vertex $v\in P$. Take the triangle $(uvw)$ contained in $P$ and formed by $v$ and its neighbors $u$ and $w$. Let $z$ be the second neighboring vertex of $w$ other than $v$. We are going to transform the triangle $(uvw)$ into another one $(uv'w)$ by shifting $v$ along the parallel line to $(uw)$ passing through $v$, such that the new vertex $v'$ lies in the line $(wz)$. This induces a global transformation keeping the rest of the polygon $P\setminus (uvw)=P\setminus (uv'w)$ unchanged (see Figure~\ref{fig:pak_ex}). Remark that we have obtained a polygon of $(n-1)$ faces and also that this transformation is continuous and piecewise linear. In fact, the triangles $(uvw)$ and $(uv'w)$ have the same area since we are not modifying neither the base of the triangle (lying in $(uw)$) nor the height (because $v$ and $v'$ lie in the same parallel line to $(uw)$). Thus, this is also a volume-preserving map.
\begin{figure}[ht]
	\centering
		\begin{tikzpicture}[scale=0.9, xscale=0.8]
		\begin{scope}[shift={(-5,2)}]
				\node (u) at (-2,1) {}; \node[left] at (u) {$u$};
				\node (v) at (-1,2) {}; \node[above] at (v) {$v$};
				\node (w) at (2,1) {}; \node[right] at (w) {$w$};
		 		\node (vp) at (3,2) {}; \node[above] at (vp) {$v'$};  \draw[fill=black] (vp) circle (0.3ex);
				\node (z) at (1,0) {}; \node[below] at (z) {$z$};
				\node (t) at (-1,0) {}; \node[below] at (t) {$t$};
		 		\node (tp) at (0.5,-0.5) {}; 
				\draw[dashed] (-2,2) -- (3.5,2); \draw[dashed] (tp) -- (3.5,2.5);
		 		\draw[fill=blue!20,semithick] (u.center) -- (v.center) -- (w.center) -- (z.center) -- (t.center) -- (u.center); 					\draw  (u.center) -- (w.center);
		\end{scope};
		
		\draw[->] (-1,3) -- (1,3); 
		
		\begin{scope}[shift={(4,2)}]
				\node (u) at (-2,1) {}; \node[left] at (u) {$u$};
				\node (w) at (2,1) {}; \node[right] at (w) {$w$};
		 		\node (vp) at (3,2) {}; \node[above] at (vp) {$v'$};
				\node (z) at (1,0) {}; \node[below] at (z) {$z$};
				\node (t) at (-1,0) {}; \node[below] at (t) {$t$};
		 		\node (tp) at (0.5,-0.5) {}; 
		 		\draw[fill=blue!20,semithick] (u.center) -- (vp.center) -- (w.center) -- (z.center) -- (t.center) -- (u.center); 				\draw  (u.center) -- (w.center);
		\end{scope};
		
		\begin{scope}[shift={(-5,-2)}]
				\node (u) at (-2,1) {}; \node[left] at (u) {$u$};
				\node (w) at (2,1) {};
		 		\node (vp) at (3,2) {}; \node[above] at (vp) {$v'$};
				\node (z) at (1,0) {}; \node[below] at (z) {$z$};
				\node (t) at (-1,0) {}; \node[below] at (t) {$t$};
		 		\node (tp) at (0.5,-0.5) {}; \node[below] at (tp) {$t'$};  \draw[fill=black] (tp) circle (0.3ex);
				\draw[dashed] (-2,1/3) -- (1,-2/3); \draw[dashed] (0,-1) -- (3.5,2.5);
		 		\draw[fill=blue!20,semithick] (u.center) -- (vp.center) -- (w.center) -- (z.center) -- (t.center) -- (u.center); 
				\draw  (u.center) -- (z.center);
		\end{scope};
		
		\draw[->] (-1,-1) -- (1,-1);
		
		\begin{scope}[shift={(4,-2)}]
				\node (u) at (-2,1) {}; \node[left] at (u) {$u$};
				\node (w) at (2,1) {};
		 		\node (vp) at (3,2) {}; \node[above] at (vp) {$v'$};
				\node (z) at (1,0) {}; \node[below] at (z) {$z$};
		 		\node (tp) at (0.5,-0.5) {}; \node[below] at (tp) {$t'$};
		 		\draw[fill=blue!20,semithick] (u.center) -- (vp.center) -- (w.center) -- (z.center) -- (tp.center) -- (u.center); 
				\draw  (u.center) -- (z.center);
		\end{scope};
		\end{tikzpicture}
	\caption{Illustration of the Monge map described by Pak.}\label{fig:pak_ex}
\end{figure}
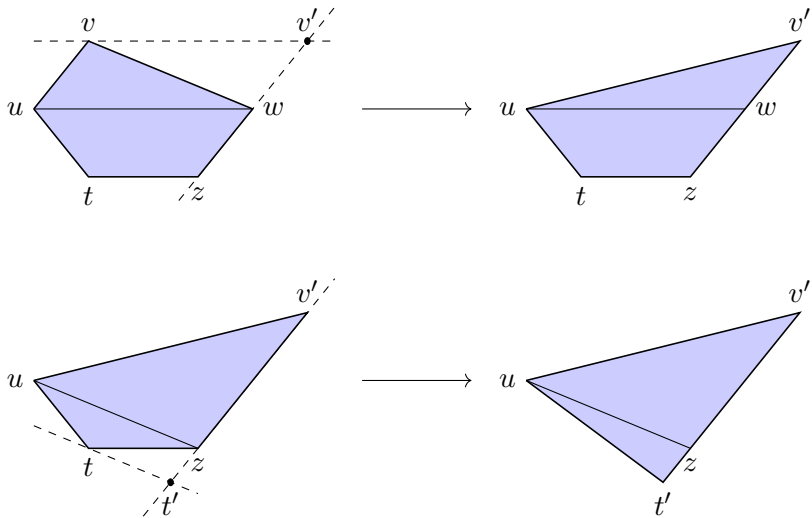

It is worth noticing that the previously described transformation respects the rationality of the polygon. Repeating this procedure $n-3$ times, we can transform $P$ on to a triangle $\triangle_1$ of equal area by a Monge map $\zeta_1:P\to\triangle_1$.

Let $Q$ be a polygon with same area of $P$. Repeating the previous procedure on $Q$, we obtain a triangle $\triangle_2$ and a Monge map $\zeta_2:Q\to\triangle_2$. As $\triangle_1$ and $\triangle_2$ have the same area by hypothesis, there exists a volume-preserving linear map $\phi:\triangle_1\to\triangle_2$. Thus, by composing the above functions we obtain a Monge map $\zeta=\zeta_2^{-1}\circ\phi\circ\zeta_1:P\to Q$ between $P$ and $Q$.

\begin{rem}
    The previous procedure is not unique and depends on many choices during the process.
\end{rem}

\subsubsection{Non-convex case}\label{subsec: canonical_no_convex}
In the non-convex case, the idea is first to decompose each polyhedron in the same number of convex pieces $P=P_1\cup\cdots\cup P_n$ and $Q=Q_1\cup\cdots\cup Q_n$ with respective volumes $\alpha_i = \vol_d(P_i)$ and $\beta_j = \vol_d(Q_j)$. Then, we can produce partitions of $P$ and $Q$ with second subdivisions $P=\bigcup_{i,j=1}^n P_{i,j}$ and $Q=\bigcup_{i,j=1}^n Q_{i,j}$ such that $\vol_d(P_{i,j})=\vol_d(Q_{i,j})=\alpha_i\beta_j$, for any $i,j=1,\ldots,n$. Such a subdivision is detailed in~\cite[Sec.~1, p.~2]{HenriquesPak04}. Moreover, such a decomposition can be made in order to preserve rationality. Applying the previous theorem at each pair of convex pieces, the result holds. Note that this proof is essentially constructed by using two allowed transformation rules: scissor-congruence and decomposition.

It is worth noticing that the previous results are still non-constructive in the general case, and this fact comes entirely from the convex case in dimension strictly greater than two.

\subsection{The general case}

Now, let us come back to the original setting, where the polyhedra $P$ and $Q$ are equipped with two piecewise algebraic volume forms $\omega_1$ and $\omega_2$ verifying $$\int_{P_1}\omega_1=\int_{P_2}\omega_2.$$

Up to our knowledge, the following result due to G.~Kuperberg in~\cite[Thm.~3]{Kuperberg96} is the most general result where non-canonical volume forms are considered. This result is formulated and proved in~\cite[Theorem~1]{HenriquesPak04} in a more suitable way with respect to our problem.

\begin{thm}\label{thm:vol_PL}
    Let $M_1,M_2\subset\RR^d$ be two \PL-manifolds, possibly with boundary, which are \PL-homeomorphic and equipped with piecewise constant volume forms $\omega_1$ and $\omega_2$. Suppose that $M_1$ and $M_2$ have equal volume, i.e. $\int_{M_1}\omega_1=\int_{M_2}\omega_2$. Then there exists a volume-preserving \PL-homeomorphism $f:M_1\to M_2$, in particular $f^*(\omega_2)=\omega_1$.
\end{thm}

\begin{rem}
    In fact, Pak and Henriques only give a sketch of proof of the above theorem, saying that it follows more or less directly from the canonical volume form case. However, certain parts need to be detailed, in particular the one analogous to the construction named in~Section~\ref{subsec: canonical_no_convex}, and dealing with two simplicial decompositions for $P$ and $Q$ by pieces $\{P_{i,j}\}_{i,j=1}^n$ and $\{Q_{i,j}\}_{i,j=1}^n$ such that $\int_{P_{i,j}}\omega_1=\int_{Q_{i,j}}\omega_2$ for any $i=1,\ldots,n$.
    
    From our point of view, the above equality between values of integrals express exactly a equality between periods! This is in fact one of the main points of the Kontsevich-Zagier conjecture and other related problems in periods theory: to compare periods between them from integral expressions~\cite[Sec.~1.2]{KonZag01}.
\end{rem}

Looking at the hypotheses of Theorem~\ref{thm:vol_PL}, we first remark that they are very strong. As long as we are looking for \emph{global} continuous, piecewise-linear, volume-preserving maps, one can not avoid the \PL-homeomorphism condition. With respect to our problem, we can not guarantee that this condition is satisfied in general. Moreover, the volume forms which appear from our construction are a priori non piecewise constant.\\

The first condition can be removed by restricting our attention to convex polyhedra. Indeed, we have the following:

\begin{lem}[{\cite[Lemma~1.1]{HenriquesPak04}}]
    For any two convex polyhedra $P,Q\subset\RR^d$ there exists a \PL-homeomorphism $f:P\to Q$.
\end{lem}

In this direction, we obtain a more satisfying result:

\begin{thm}\label{thm:vol_PL_convex}
    Let $P,Q\subset\RR^d$ be two convex polyhedra equipped with piecewise constant volume forms $\omega_1$ and $\omega_2$. Suppose that $P$ and $Q$ have equal volume, i.e. $\int_{P}\omega_1=\int_{Q}\omega_2$. Then there exists a volume-preserving \PL-homeomorphism $f:P\to Q$, in particular $f^*(\omega_2)=\omega_1$.
\end{thm}


\section{Conclusion}\label{sec:KZconj_conclu}

We have presented two geometric stronger versions of the Kontsevich-Zagier period conjecture which take place in different settings (Grothendieck rings, K-theory, \PL-geometry, volumes and forms on convex polyhedra,...) where there already exists some similar tools to work on possible schemes of proof and obstructions.

The latter open different ways to approach this problem, generalizing some of the previous results in more general settings for constructing a complete proof. On the other hand, the considerable quantity of obstructions which exist in other similar scissor-congruence problems give some evidence to believe that the Kontsevitch-Zagier period conjecture could be, in fact, false.

The GKZ-conjecture %
presents an algebraic setting to work on the construction of obstructions and geometric invariants. This could lead to possible counterexamples, in the spirit of other cutting-and-pasting problems. The main goal here is to develop a Dehn-like invariant for general compact semi-algebraic sets.

Concerning
the \PL-GKZ-conjecture, even if the type of result that Theorem~\ref{thm:vol_PL_convex} represents is exactly in the spirit of our problem, the gap between this result and a complete answer in the \PL~case is large due to the following reasons:
\begin{itemize}
    \item We think that the most difficult problem to deal with is to extend the previous theorem for volume forms in the algebraic class.
    
    \item We would like to avoid the convexity assumption. In this case, the situation seems more tractable following the same strategy as for canonical volume forms. Indeed, the setting of our problem allows scissors-congruences between the polyhedra and not only global volume-preserving \PL-homeomorphism: it seems reasonable that we could give decompositions of non-convex polyhedra in convex parts with a predefined volume and then to "compare" these parts one by one.
\end{itemize}

Due to the connection between the Kontsevich-Zagier period conjecture,
the GKZ-conjecture and the \PL-GKZ-conjecture, the above describes the main problems to be solved in order to advance in the knowledge of relations between periods from both new settings.

Recently, J.~Commelin, P.~Habegger and A.~Huber have obtained in~\cite{CHH:exponentialperiodsI} a ``semi-canonical representation'' type of result in the context of exponential periods and o-minimality. Using this, it would be interesting to understand the previous problems and obstructions in the case of exponential periods.


\section*{Acknowledgments}

The authors would like to strongly thank Professors Michel Waldschmidt and Pierre Cartier for all the very helpful discussions and ideas. The second author would like to thank the first one for his constant encouragement and support during his Ph.D.

\bigskip
\bigskip
\bibliographystyle{alpha}
\bibliography{biblio}

\end{document}